\documentclass{elsarticle}

\usepackage[all,cmtip]{xy}
\usepackage{array}
\usepackage{amsmath, amssymb, amsthm}
\usepackage{graphics}
\usepackage{algorithm}
\usepackage{algorithmic}
\usepackage{booktabs}
\usepackage[dvipsnames]{xcolor}
\usepackage{cancel}
\usepackage{caption}
\usepackage{subcaption}
\usepackage{pdflscape}
\usepackage[normalem]{ulem}

\def\rmd{\textup{d}}

\def\Hom{\textup{Hom}}

\def\CC{\mathbb{C}}
\def\KK{\mathbb{K}}
\def\PP{\mathbb{P}}
\def\QQ{\mathbb{Q}}
\def\RR{\mathbb{R}}

\def\ZZ{\mathbb{Z}}

\def\CCC{\mathcal{C}}
\def\DDD{\mathcal{D}}
\def\FFF{\mathcal{F}}

\def\SSS{\mathcal{S}}

\def\bfb{{\boldsymbol{b}}}
\def\bfc{{\mathbf{c}}}
\def\bfd{{\mathbf{d}}}
\def\bff{{\boldsymbol{f}}}

\def\bfp{{\boldsymbol{p}}}
\def\bfq{{\boldsymbol{q}}}

\def\bfx{{\boldsymbol{x}}}

\def\bfv{{\boldsymbol{v}}}

\def\bfP{{\mathbf{P}}}

\def\diag{\mathrm{diag}}

\def\bfI{{\boldsymbol{I}}}

\def\bfA{{\boldsymbol{A}}}
\def\bfC{{\boldsymbol{C}}}

\def\bfM{{\boldsymbol{M}}}

\def\bfR{{\boldsymbol{R}}}
\def\bfS{{\boldsymbol{S}}}
\def\bfT{{\boldsymbol{T}}}

\newtheorem{theorem}{Theorem}
\newtheorem{corollary}[theorem]{Corollary}
\newtheorem{proposition}[theorem]{Proposition}

\theoremstyle{definition}

\newtheorem{example}{Example}
\newtheorem{remark}{Remark}

\graphicspath{{./figs/}}

\begin{document}

\begin{frontmatter}
\title{Affine equivalences of surfaces of translation and minimal surfaces,\\ and applications to symmetry detection and design}

\author[a]{Juan Gerardo Alc\'azar\fnref{proy}}
\ead{juange.alcazar@uah.es}
\author[b]{Georg Muntingh}
\ead{Georg.Muntingh@sintef.no\fnref{c2t}}

\address[a]{Departamento de F\'{\i}sica y Matem\'aticas, Universidad de Alcal\'a,
E-28871 Madrid, Spain}
\address[b]{SINTEF Digital, PO Box 124 Blindern, 0314 Oslo, Norway}

\fntext[proy]{Partially supported by the grant PID2020-113192GB-I00 (Mathematical Visualization: Foundations, Algorithms and Applications) from the Spanish MICINN. Juan G. Alc\'azar is also a member of the Research Group {\sc asynacs} (Ref. {\sc ccee2011/r34}).  }

\fntext[c2t]{This project has received funding from the European Union's Horizon 2020 research and innovation programme under grant agreement No 951956.}

\begin{abstract}
We introduce a characterization for affine equivalence of two surfaces of translation defined by either rational or meromorphic generators. In turn, this induces a similar characterization for minimal surfaces. In the rational case, our results provide algorithms for detecting affine equivalence of these surfaces, and therefore, in particular, the symmetries of a surface of translation or a minimal surface of the considered types. Additionally, we apply our results to designing surfaces of translation and minimal surfaces with symmetries, and to computing the symmetries of the higher-order Enneper surfaces.
\end{abstract}

\end{frontmatter}

\section{Introduction}
\emph{Surfaces of translation}, also called \emph{translational surfaces} (c.f. \cite{VL17}) are surfaces generated by sliding one space curve along another space curve. Due to their simplicity, these surfaces are used in Computer-Aided Geometric Design. In particular, any two intersecting curves are interpolated by the surface of translation generated by these curves, and the bilinear Coons patch interpolating four boundary curves can be expressed as a weighted linear combination of the translational surfaces generated by the two curves at each corner. Efficient algorithms for computing $\mu$-bases and implicitization are known \cite{WG18a}, \cite{WG18b}.

\emph{Minimal surfaces} (c.f. \cite{Ohdenal, VL17} and \cite[Chapters 16 and 22]{Gray}) are surfaces whose mean curvature is identically zero. It was already known by Sophus Lie that such surfaces are also surfaces of translation generated by complex conjugated curves. Minimal surfaces have the remarkable property of spanning a given space curve with minimal area. Because of this property, they arise frequently in nature, for instance in soap films, and are useful in architecture. In addition, minimal surfaces have applications across the sciences, for instance in general relativity, molecular biology, and material science. 

Two surfaces are \emph{affinely equivalent} if there exists a nonsingular affine map transforming one of the surfaces onto the other. Recognizing affine equivalence is of interest in computer-aided geometric design, in computer vision and in pattern recognition. Two notable instances of affine equivalence are similarity and symmetry: two surfaces are similar when they are the same surface up to a scaling, translation, rotation and reflection; a surface is symmetric when it is invariant under a nontrivial isometry. 

Recently there have been several papers introducing methods for recognizing projective equivalences, affine equivalences and symmetries for rational curves and surfaces. For rational curves the problem can be considered as essentially solved; see for instance \cite{AHM14, AHM15b, HJ18}. For rational surfaces the problem is more complicated, and the general case is still unsolved. However,
progress has been made in special cases. Involutions of polynomially parametrized surfaces are addressed in~\cite{AHM15}, while symmetries of canal surfaces and Dupin cyclides were investigated in~\cite{ADM}. In \cite{HJ18-2}, an algorithm for computing projective and affine equivalences for the case of rational parametrizations without projective base points is given. Affine equivalences for ruled surfaces are considered in \cite{AQ}. Projective equivalences of ruled surfaces are studied in \cite{BLV}, where some aspects of the case of implicit algebraic surfaces are also addressed.

In this paper we introduce a characterization for affine equivalence of two surfaces of translation, and therefore also of two minimal surfaces. We focus on rational surfaces, although our results extend also to the meromorphic case. For the rational case, our results give rise to algorithms for detecting affine equivalence. We also apply our method to detecting symmetries in the considered types of surfaces, and to designing symmetric surfaces of translation and minimal surfaces. Functionality for generating such results, as well as all the examples in this paper, are provided in Python and Sage, available as a GitHub repository \cite{Muntingh}; the repository also includes the files corresponding to two examples carried out in Maple.

\bigskip
\noindent
{\bf Acknowledgements:} We thank the anonymous referees for their comments on a previous version of the paper, which helped to improve the final version.

\section{Background}

\subsection{Affine equivalences and symmetries} \label{subsec-affine}

A nonsingular \emph{affine map $\bff$} of $\RR^n$ takes the form $\bff(\bfx) = \bfM \bfx + \bfb$, with $\bfb\in \RR^n$ a vector and $\bfM\in \RR^{n\times n}$ a nonsingular matrix. If $\bfM$ is orthogonal, i.e., $\bfM\bfM^\textup{T} = \bfI$, then $\bff$ defines an \emph{isometry}. Given two surfaces $\SSS_1,\SSS_2$, we say that $\SSS_1,\SSS_2$ are \emph{affinely equivalent} if there exists a nonsingular affine map $\bff$ such that $\bff(\SSS_1)=\SSS_2$. In this case we also say that $\bff$ is an \emph{affine equivalence} between $\SSS_1,\SSS_2$; similarly for two curves $\CCC,{\mathcal D}$. If $\SSS_1=\SSS_2$ and $\bff(\bfx) = \bfM \bfx + \bfb$ with $\bfM$ orthogonal, then we say that $\bff$ is a \emph{symmetry} of the surface; similarly for a curve $\CCC$. Although we will consider both real and complex curves, we will only consider affine equivalences and symmetries that are real. The identity map $\bff = \text{id}_{\RR^n}$ is referred to as the \emph{trivial} isometry/symmetry. A curve or surface is called \emph{symmetric} if it has a nontrivial symmetry. Notable symmetries are \emph{planar symmetries} (i.e., reflections in a plane), \emph{axial symmetries} (i.e., rotations about a line), \emph{central symmetries} (i.e., symmetries with respect to a point), and \emph{rotoreflections} (i.e., composition of a rotation about a line and a reflection in a plane perpendicular to this line). Special cases of axial symmetries are the \emph{half-turn} (rotation by angle $\pi$) and the \emph{quarter-turn} (rotation by angle~$\pm \pi/2$). 

For further information on nontrivial symmetries of Euclidean space, see \cite{Coxeter69} and \cite[\S 2]{AHM15}.

\subsection{Surfaces of translation}
In this section we introduce two closely related concepts of surfaces of translation, namely those generated by real curves and those generated by complex conjugated curves. 

\subsubsection{Averaging operator}
Let $\KK$ be a field and $\KK^n$ the corresponding $n$-dimensional affine space over $\KK$. In this paper we consider $\KK = \RR, \CC$. Following \cite{VL17}, we equip $\KK^n$ with the binary operation $\oplus$ defined by taking the average, i.e.,
\begin{equation}\label{eq:AverageOperator}
\oplus: \KK^n \times \KK^n \longrightarrow \KK^n,\qquad \bfp\oplus\bfq := \frac{\bfp + \bfq}{2}.
\end{equation}
By abuse of notation, we can also consider $\oplus$ as a binary operation on the space of rational (or meromorphic) parametrizations,
\begin{equation}\label{eq:AverageOperator2}
\begin{gathered}
\oplus: \Hom(\KK, \KK^n) \times \Hom(\KK, \KK^n) \longrightarrow \Hom(\KK^2, \KK^n),\ \\
(f\oplus g) (t,s) := f(t)\oplus g(s),
\end{gathered}
\end{equation}
where $\Hom(\KK^m, \KK^n)$ denotes the space of rational (or meromorphic) maps $\KK^m\dashrightarrow \KK^n$.

Note that composition with an affine map $\bff(\bfx) = \bfM \bfx + \bfb$ is distributive with respect to $\oplus$, i.e.,
\begin{equation}\label{distrib} 
\bff \circ (\bfc_1\oplus\bfc_2) = (\bff \circ \bfc_1) \oplus (\bff \circ \bfc_2),
\end{equation}
because, for any $t,s$, 
\[ \bff \circ (\bfc_1 \oplus \bfc_2) (t,s)
 = \frac{\bfM\bfc_1(t)+\bfb+ \bfM\bfc_2(s)+\bfb}{2} = (\bff\circ \bfc_1) (t) \oplus (\bff\circ \bfc_2) (s). \]

\subsubsection{Real surfaces of translation}
From its definition \eqref{eq:AverageOperator} on the pair of affine spaces $\KK^n$, the operator $\oplus$ can restrict to pairs of space curves $\CCC_1,\CCC_2 \subset \KK^3$. In this manner we arrive at a set $\SSS = \CCC_1\oplus \CCC_2$ of averages of all pairs of points in $\CCC_1, \CCC_2$, called the \emph{surface of translation generated by} $\CCC_1, \CCC_2$. 
We will say that $(\CCC_1,\CCC_2)$ is a \emph{generator pair} of $\SSS$. In particular, $\SSS$ contains two families of congruent curves, which are translated copies of the curves $\CCC_1,\CCC_2$, scaled by a factor $\frac{1}{2}$. 

Parametrizations of the curves $\CCC_1, \CCC_2$ induce a corresponding parametrization of the surface $\SSS$ through \eqref{eq:AverageOperator2}. In this paper we consider (not necessarily distinct) parametrizations
\begin{equation}\label{eq:CurveParametrizations}
\bfc_1(t)=\big(x_1(t),y_1(t),z_1(t)\big), \qquad \bfc_2(s)=\big(x_2(s),y_2(s),z_2(s)\big), \qquad t,s\in \KK,
\end{equation}
where $x_1,x_2,y_1,y_2,z_1,z_2$ are rational or meromorphic functions with coefficients in $\KK$, parametrizing the space curves $\CCC_1,\CCC_2 \subset \KK^3$. The surface $\SSS = \CCC_1\oplus \CCC_2$ then has the corresponding parametrization
\begin{equation}\label{eq:TranslationalSurfaceParametrization}
\bfP := \bfc_1\oplus \bfc_2: \KK^2 \dashrightarrow \SSS\subset \KK^3, \qquad
(t,s) \longmapsto \bfc_1(t) \oplus \bfc_2(s).
\end{equation}
In particular, if $\bfc_1,\bfc_2$ are rational$/$meromorphic, then $\bfP$ is rational$/$meromorphic as well.

The focus of this paper is on \emph{real} surfaces of translation, due to their applicability in computer-aided geometric design. These are defined as nondegenerate (i.e., two-dimensional) real surfaces that are the real part $\Re(\SSS) = \SSS \cap \RR^3$ of $\SSS = \CCC_1 \oplus \CCC_2$ as defined above.

Can we obtain real parametrizations of real surfaces of translation? When $\bfc_1, \bfc_2$ have real coefficients, then $(\bfc_1 \oplus \bfc_2) (t,s)$ will be real for any real parameters $t, s$. However, the real surface of translation can contain an additional 1-dimensional singular locus coming from complex conjugate points on $\CCC_1,\CCC_2$ with complex conjugate parameters.

If $\bfc_1, \bfc_2$ have complex coefficients, then it is in general difficult (and not always possible) to provide a parametrization of $\SSS$ with real coefficients. However, a real parametrization of $\Re(\SSS)$ can be obtained in the special case of complex conjugate curves $\CCC_1, \CCC_2$. To be precise, let $\overline{\phantom{z}}$ denote the map that takes the complex conjugation of complex numbers and complex vectors (component-wise), as well as their sets (element-wise). Moreover, for scalar- or vector-valued rational (or meromorphic) functions $\bfc$, we let $\overline{\bfc}$ denote the function resulting from conjugating the coefficients of $\bfc$.

With $\bfc := \bfc_1 = \overline{\bfc}_2$ and $\KK=\CC$ in \eqref{eq:CurveParametrizations}, we have conjugated parametric curves
\begin{align} 
\bfc & = (x,y,z):U\subset \CC \longrightarrow \CCC\subset \CC^3,\\
\overline{\bfc} & = (\overline{x},\overline{y},\overline{z}) : \overline{U}\subset \CC \longrightarrow \overline{\CCC}\subset \CC^3.
\end{align}
In this case $\bfc \oplus \overline{\bfc}$, restricted to complex conjugate parameters $t = u+iv$ and $s = u - iv$, provides a real parametrization in terms of $u$ and $v$. More precisely, precomposing $\bfc\oplus\overline{\bfc}$ with the embedding $\iota: (u, v) \longmapsto (u + iv, u - iv)$ of $\RR^2$ into $\CC^2$
yields a parametrization 
\begin{align}\label{eq:re2}
\bfP := \big(\bfc \oplus \overline{\bfc}\big) \circ \iota : \RR^2 & \longrightarrow \Re(\SSS)\subset \RR^3,\\
(u,v) & \longmapsto 
 \frac{\bfc(u+iv) + \overline{\bfc}(u - iv)}{2}
 = \bfc(t)\oplus \overline{\bfc}(\overline{t}),\notag
\end{align}
of (part of) the real part of $\SSS$. Note that \eqref{eq:re2} is a parametrization with real coefficients. This case will be considered in detail in Section \ref{sec:MinimalSurfaces1}.

\subsubsection{Multitranslational surfaces}
Note that a surface of translation does not have a unique generator pair. Indeed, if $\tau_{\bfv}: \bfx \longmapsto \bfx + {\bfv}$ denotes the translation by a vector ${\bfv}$ and $(\CCC_1,\CCC_2)$ is a generator pair of $\SSS$, then
\begin{equation}\label{eq:generatorpairs}
(\DDD_1,\DDD_2) := \big(\tau_{\bfv}(\CCC_i), \tau_{-\bfv}(\CCC_j)\big),\qquad \{i,j\} = \{1,2\},\qquad {\bfv}\in \CC^3,
\end{equation}
also has the property that $\DDD_1\oplus \DDD_2 = \SSS$. Indeed, $\DDD_1\oplus \DDD_2$ is parametrized by \begin{equation}\label{expla}(\bfc_i+\bfv)\oplus(\bfc_j-\bfv)=\dfrac{(\bfc_i+\bfv)+(\bfc_j-\bfv)}{2}=\dfrac{\bfc_i+\bfc_j}{2}=\bfc_i\oplus \bfc_j,
\end{equation}
which, since $i\neq j$, parametrizes $\SSS$. Notice that if $\bfv$ is nonreal and the parametrizations of $\CCC_1,\CCC_2$ are real or complex conjugate, then the parametrizations of $\DDD_1,\DDD_2$ are no longer real or complex conjugate. Nevertheless, because of \eqref{expla}, their parametrizations still generate a real parametrization of the same real surface of translation $\SSS$.

However, some surfaces of translation $\SSS$ possess more exotic alternative generator pairs for which \eqref{eq:generatorpairs} does not hold. In that case we say that $\SSS$ is \emph{multitranslational}. This definition includes the class of double translational surfaces considered by Sophus Lie and Poincar\'e (c.f. \cite[\S 1]{Chern}).

We will require the following technical assumptions. The first is that the generator curve parametrizations $\bfc_1,\bfc_2$ are \emph{proper}, i.e., injective for all but finitely many values of the parameter. Secondly, we will assume that $\SSS$ is not multitranslational. Thirdly, we will assume that $\CCC_1,\CCC_2$ are not planar curves contained in the same plane or in parallel planes; in that case $S$ would be a plane.

\subsection{Minimal surfaces}\label{sec:MinimalSurfaces1}
\emph{Minimal surfaces} are surfaces with constant vanishing mean curvature. Minimal surfaces are sometimes defined as surfaces of the smallest area spanned by a given closed space curve, with illustrative physical examples provided by soap films spanning a given wireframe. For us, however, the most relevant fact about minimal surfaces is that they are surfaces of translation with complex conjugate generating pair (c.f. \cite[\S 5.4]{Struik}). In particular, this follows from a classic representation of minimal surfaces, called the \emph{Weierstrass form} of the surface. Let $U\subset \CC$ be a simply-connected domain, and let $t_0$ be a point in the interior of $U$. Weierstrass proved (see Theorem 1 in page 112 of \cite[\S 3.3]{DHKW})  
that any nonplanar minimal surface $\SSS$ defined over a simply-connected parameter domain can be parametrized as the real part
\begin{equation}\label{re1}
\bfP(u,v)=\Re\big(\bfc(t)\big)=\Re\big(\bfc(u+iv)\big),
\qquad (u,v)\in \RR^2,
\end{equation}
of the complex curve $\CCC \subset \CC^3$ with parametrization
\[ \bfc = (x,y,z):U\subset \CC \longrightarrow \CC^3,\]
where, with $i^2 = -1$,
\begin{equation}\label{weierstrass}
\bfc(t) = \left( \int_{t_0}^t f(\gamma)\frac{1-g(\gamma)^2}{2}\rmd\gamma, i\int_{t_0}^t f(\gamma)\frac{1+g(\gamma)^2}{2}\rmd\gamma, \int_{t_0}^t f(\gamma)g(\gamma) \rmd\gamma\right).
\end{equation}
Here $f$ is holomorphic and $g$ is meromorphic such that $fg^2$ is holomorphic in $U$. This condition implies that, within $U$, any pole of $g$ of order $k$ is located at a zero of $f$ of order at least $2k$. In that case the integrands in \eqref{weierstrass} are holomorphic in $U$, and the integrals are well-defined for any $t\in U$. In several cases the functions $f,g$ are real functions which are real analytic on an interval containing $t_0$, and $U$ is an open subset of $\CC$ where the components of $\bfc(t)$ can be analytically extended; some examples can be found, for instance, in Chapter~22 of \cite{Gray}. A straightforward calculation shows that $\bfc$ is an \emph{isotropic curve}, i.e.,
\[ x'(t)^2 + y'(t)^2 + z'(t)^2 = 0. \]

In the context of minimal surfaces, $\bfc$ is sometimes called a \emph{minimal curve}. We will say that $\bfc$ \emph{generates} the minimal surface. Thus, any minimal surface $\SSS$ can be parametrized as \eqref{re1}, where $\bfc$ generates $\SSS$. Notice that this is a \emph{real} parametrization of $\SSS$, i.e., $\bfP(t,s)$ has real coefficients.
Since the real part of a complex number is equal to the average of itself and its complex conjugate, the parametrization \eqref{re1} takes the form \eqref{eq:re2}, and $\SSS=\CCC\oplus \overline{\CCC}$ is a surface of translation with complex conjugate generator pair $(\CCC,\overline{\CCC})$.

\section{Translational surfaces}
\subsection{Detecting affine equivalences and symmetries}
In this subsection we will initially assume that the generator pairs are rational, and later we will provide an observation (Remark \ref{extension-meromorphic}) that allows us to extend our results to the meromorphic case. We start with the main result. 

\begin{theorem} \label{th-fundam-trans}
Let $\SSS_1=\CCC_1\oplus \CCC_2$, $\SSS_2=\DDD_1\oplus \DDD_2$ be two rational surfaces of translation, which are not multitranslational. Then $\bff(\bfx) = \bfM \bfx + \bfb$ is an affine equivalence between $\SSS_1$ and $\SSS_2$ if and only if 
\begin{enumerate}
\item there exists ${\bfv}\in \KK^3$ such that $\bff(\CCC_1)=\tau_{\bfv}(\DDD_1)$ and $\bff(\CCC_2)=\tau_{-{\bfv}}(\DDD_2)$, or
\item there exists ${\bfv}\in \KK^3$ such that $\bff(\CCC_1)=\tau_{\bfv}(\DDD_2)$ and $\bff(\CCC_2)=\tau_{-\bfv}(\DDD_1)$.
\end{enumerate}	
\end{theorem} 

\begin{proof}
Let $\bfc_1,\bfc_2,\bfd_1,\bfd_2$ be the parametrizations of $\CCC_1,\CCC_2,\DDD_1,\DDD_2$. Let $\tilde\CCC_1,\tilde\CCC_2$ be the curves defined by the parametrizations $\tilde{\bfc}_1 := \bff \circ \bfc_1$, $\tilde{\bfc}_2 := \bff \circ \bfc_2$.

``$\Longrightarrow$'': Since $f$ is an affine equivalence between $\SSS_1$ and $\SSS_2$, any point of $\SSS_2$ can be written as $\big(\tilde{\bfc}_1(t)+\tilde{\bfc}_2(s)\big)/2$, implying $\SSS_2=\tilde\CCC_1\oplus \tilde\CCC_2$. Thus, $(\tilde\CCC_1,\tilde\CCC_2)$ and $(\DDD_1,\DDD_2)$ are both generator pairs of $\SSS_2$. Since $\SSS_2$ is not multitranslational by hypothesis, the result follows. 

``$\Longleftarrow$'': We just prove Case 1; Case 2 is analogous. Since $(\DDD_1,\DDD_2)$ is a generator pair of $\SSS_2$, then $\big(\bff(\CCC_1), \bff(\CCC_2) \big) = \big(\tau_{\bfv}(\DDD_1),\tau_{-\bfv}(\DDD_2)\big)$ is also a generator pair of $\SSS_2$. Since $\bff$ is distributive with respect to $\oplus$, we get
\[ \SSS_2 = \bff(\CCC_1)\oplus \bff(\CCC_2) = \bff(\CCC_1\oplus \CCC_2) = \bff(\SSS_1), \]
which proves the claim. 
\end{proof}

Writing $\bfb_1 = \bfb - \bfv$ and $\bfb_2 = \bfb + \bfv$, we can rephrase Theorem \ref{th-fundam-trans} as follows.

\begin{corollary}\label{cor1}
Let $\SSS_1=\CCC_1\oplus \CCC_2$, $\SSS_2=\DDD_1\oplus \DDD_2$ be two rational surfaces of translation, which are not multitranslational. Then $\SSS_1$ and $\SSS_2$ are affinely equivalent if and only if there exist two nonsingular affine maps (with identical matrix)
\[ \bff_1(\bfx)=\bfM\bfx+\bfb_1, \qquad \bff_2(\bfx)=\bfM \bfx+\bfb_2\]
such that either
\begin{enumerate}
\item $\bff_1(\CCC_1)=\DDD_1,\ \bff_2(\CCC_2)=\DDD_2$, or
\item $\bff_1(\CCC_1)=\DDD_2,\ \bff_2(\CCC_2)=\DDD_1$,
\end{enumerate}
In that case $\bff(\bfx)=\bfM \bfx+\bfb_1\oplus \bfb_2$ is an affine equivalence between $\SSS_1$ and $\SSS_2$. 
\end{corollary}

Thus, Corollary \ref{cor1} allows to transfer the affine equivalence detection problem from surfaces of translation to their generating space curves. In order to do this, we recall here the following result from \cite{HJ18}; the result uses the fact that the only birational transformations of the complex line are the \emph{M\"obius transformations} \cite{SWP}, i.e., rational functions 
\begin{equation}\label{eq:Moebius}
\varphi: \CC \dashrightarrow \CC, \qquad \varphi(t) = \frac{a t + b}{c t + d}, \qquad ad - bd \neq 0.
\end{equation}

\begin{proposition}\label{thm:diagram}
Let $\CCC,\DDD\subset \CC^3$ be two rational space curves, properly parametrized by $\bfc,\bfd$. Then $\CCC,\DDD$ are affinely equivalent if and only if there exists an affine map $\bff(\bfx)=\bfM \bfx+\bfb$ and a M\"obius transformation $\varphi$ such that 
\begin{equation}\label{af}
\bff\circ\bfc=\bfd\circ\varphi,
\end{equation}
\end{proposition}

In \cite{HJ18}, it is shown how to use Proposition \ref{thm:diagram} to solve the affine equivalence problem for space curves (and in fact, for projective equivalences between rational curves in any dimension). The rough idea is that \eqref{af} leads to a polynomial system, linear in the entries of $\bfM$ and the components of $\bfb$. Some of the equations of this system can be used to write the entries of $\bfM$ and the components of $\bfb$ in terms of the parameters of the M\"obius transformation $\varphi$. When these expressions are plugged into the remaining equations, we get polynomial conditions for the parameters of the M\"obius transformation $\varphi$. Computing these parameters then leads to the affine equivalences themselves. Combining this with Corollary \ref{cor1}, we arrive at Algorithm {\tt Affine-Equiv-Trans} for solving the affine equivalence problem for surfaces of translation. 

\begin{algorithm}[t!]
\begin{algorithmic}[1]
\REQUIRE Two surfaces of translation $\SSS_1=\CCC_1\oplus\CCC_2$, $\SSS_2=\DDD_1\oplus \DDD_2$, rationally parametrized by 
$\bfP_1, \bfP_2$ as in \eqref{eq:TranslationalSurfaceParametrization} in the real case or \eqref{eq:re2} in the complex conjugate case, where the underlying generating curve pairs $\CCC_1,\CCC_2$ and $\DDD_1,\DDD_2$ are not coplanar and given by proper, rational parametrizations.
\ENSURE The affine equivalences $\bff(\bfx)=\bfM \bfx+\bfb$ between $\SSS_1$ and $\SSS_2$, or the statement that $\SSS_1$ and $\SSS_2$ are not affinely equivalent.
\FOR{$i=1,2$ and $j=1,2$}
\STATE{Determine the affine equivalences $\bff_{ij}:\CCC_i\longrightarrow \DDD_j$.}
\ENDFOR
\STATE{For any pair $(\bff_{11}(\bfx)=\bfM\bfx+\bfb_{11}, \bff_{22}(\bfx)=\bfM\bfx+\bfb_{22})$ with equal matrix,\\
 {\bf return} ``{\tt $\bff(\bfx)=\bfM \bfx+\bfb_{11}\oplus \bfb_{22}$ is an affine equiv. between $\SSS_1,\SSS_2$}}''
\STATE{For any pair $(\bff_{12}(\bfx)=\bfM\bfx+\bfb_{12}, \bff_{21}(\bfx)=\bfM\bfx+\bfb_{21})$ with equal matrix,\\
 {\bf return} ``{\tt $\bff(\bfx)=\bfM \bfx+\bfb_{12}\oplus \bfb_{21}$ is an affine equiv. between $\SSS_1,\SSS_2$}}''
\IF{no such affine equivalence pair with equal matrix is found}
\STATE{{\bf return} ``{\tt $\SSS_1$ and $\SSS_2$ are not affinely equivalent}}''
\ENDIF
\end{algorithmic}
\caption{Affine-Equiv-Trans}\label{affine-equiv-trans}
\end{algorithm}

Notice that the complexity of Algorithm {\tt Affine-Equiv-Trans} is dominated by the complexity of computing affine equivalences between the generator curves. In this regard, we refer the interested reader to the exhaustive analysis of the performance of the algorithm for computing affine equivalences between curves carried out in \cite{HJ18}. In particular, this analysis shows that the algorithm is very efficient even for high degree.

Setting ${\mathcal S}_1={\mathcal S}_2$ in Algorithm {\tt Affine-Equiv-Trans}, and requiring $\bff(\bfx)=\bfM\bfx + \bfb$ to be an isometry (i.e., $\bfM$ orthogonal), leads to an analogous algorithm for computing symmetries of translational surfaces.

\begin{remark}\label{number}
Note that in Algorithm {\tt Affine-Equiv-Trans} we need to compute the affine equivalences $\bff_{ij}: \CCC_i \longrightarrow \DDD_j$ of the four curve pairs $\{\CCC_1, \CCC_2\} \times \{\DDD_1, \DDD_2\}$.
\end{remark} 

\begin{remark}\label{extension-meromorphic}
Let $\varphi: \CC \longrightarrow \CC$ be a meromorphic function, i.e., a quotient $\varphi = \varphi_1/\varphi_2$ of two holomorphic functions. Embedding the complex plane $\CC$ as an affine chart of the complex projective line $\PP^1_\CC$ through the map $z \longmapsto [z: 1]$, the meromorphic function $\varphi$ can be extended to an analytic function 
\[ \PP^1_\CC \longrightarrow \PP^1_\CC,\qquad [t:s] \longmapsto [\varphi_1 (t/s): \varphi_2 (t/s)].\]
It follows from Liouville's theorem that every such analytic function on $\PP^1_\CC$ is rational \cite[\S 2.9]{Forster}. % Or Shafarevich, Algebraic Geometry 1, p. 24, example 1.
Hence, perhaps surprisingly, also the bi-analytic bijections on the complex projective line are M\"obius transformations. Due to this, results analogous to Theorem \ref{th-fundam-trans}, Corollary \ref{cor1} and  Proposition \ref{thm:diagram} hold for curves $\CCC, \DDD \subset \CC^3$ with proper meromorphic parametrizations $\bfc, \bfd$. However, we are unaware of the existence of algorithms for checking affine equivalence of generic meromorphic curves. Hence an analogous result to Algorithm {\tt Affine-Equiv-Trans} for meromorphic curves is not available at the moment.
\end{remark}

\begin{example}
Consider the twisted cubic curves
\[
\CCC: t\longmapsto \bfc(t) = \big(t, t^2, t^3\big), \qquad
\DDD: t\longmapsto \bfd(t) = \big(t^3, -t, t^2\big),
\]
as well as corresponding surfaces of translation $\CCC\oplus \CCC$ and $\DDD\oplus\DDD$ obtained by translating these curves along themselves (see Figure \ref{fig:twistedcubicsurfaces}). In this case Algorithm~\ref{affine-equiv-trans} simplifies, as it is only requires detecting affine equivalences $\bff: \CCC\longrightarrow \DDD$ (cf. Remark \ref{number}).

For ease of presentation we will restrict our attention to affine equivalences $\bff(\bfx)=\bfM\bfx + \bfb$ with zero constant term, i.e., $\bfb=0$. First we consider the case
where $\varphi$ as in \eqref{eq:Moebius} satisfies $d=1$. Let
\[ R=\QQ[a,b,c,m_{11},\ldots,m_{33}]\]
be the ring of polynomials with rational coefficients in the entries $m_{ij}$ of $\bfM$ and the coefficients $a, b, c$ of $\varphi$. Substituting a generic $\varphi$ (with $d=1$) together with a generic matrix $\bfM = [m_{ij}]_{i,j=1}^3$ in \eqref{af}, clearing denominators, and equating the coefficients of the monomials $1, t, t^2, \ldots$, one obtains a system of equations $\{ p_k = 0\}$, where the polynomials $p_k$ generate an ideal $I\subset R$. Using Sage with Singular as a back-end (see \cite{Muntingh}), we compute a Gr\"obner basis $G$ of $I$. After eliminating the variables $\{m_{ij}\}$, we obtain an elimination ideal in the ring $\QQ[a,b,c]$ with Gr\"obner basis
\[ \{ a^3c, a^2c^2, ac^2, b, d-1 \}. \]
It follows that $d=1$ (as imposed), $b=0$ and (since $ad-bc\neq 0$) in addition $c=0$ and $a\neq 0$, yielding the M\"obius transformations $\varphi_a(t) = at$. Substituting this back into $G$ one finds that all entries of $\bfM$ are zero, except for $m_{13}, m_{21}, m_{32}$, which satisfy $m_{21} = -a, m_{32} = m_{21}^2 = a^2$, and $m_{13} = -m_{32}^2/m_{21} = a^3$. A similar analysis with $d=0$ in \eqref{eq:Moebius} provides no new solutions.

We conclude that the affine equivalences (with zero constant term) between $\CCC$ and $\DDD$ and corresponding M\"obius transformations $\varphi$ satisfying \eqref{af} are
\[
\bff_a (\bfx) =
\begin{bmatrix}
0 & 0 & a^3\\ 
-a & 0 & 0 \\
0 & a^2 & 0
\end{bmatrix}\bfx,\qquad
\varphi_a(t) = a t,\qquad 0\neq a\in \RR.
\]
By Corollary \ref{cor1}, these are also affine equivalences between $\CCC\oplus \CCC$ and $\DDD\oplus\DDD$. Note that isometries are obtained for $a=\pm 1$.
\end{example}

\begin{figure}
    \centering
    \includegraphics[scale=1]{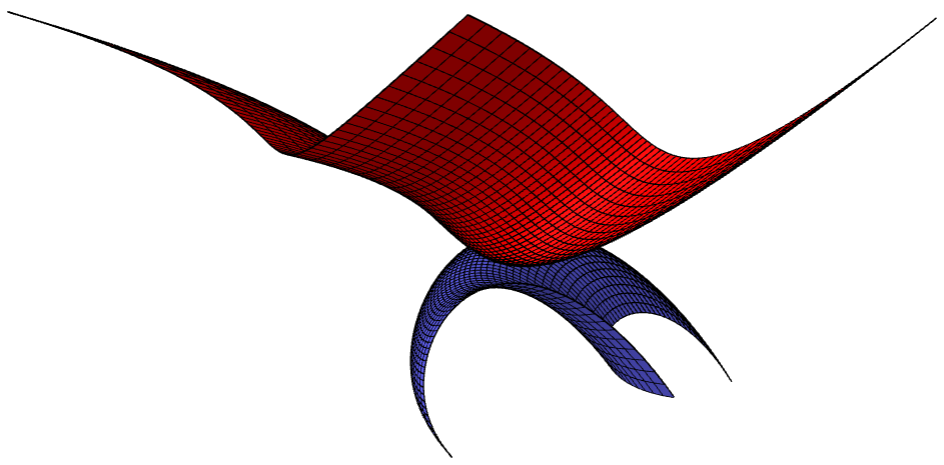}
    \caption{Isometric translational surfaces, each generated by translating a twisted cubic curve along itself.}
    \label{fig:twistedcubicsurfaces}
\end{figure}

\begin{example}\label{sec-example}
Let $S_1=\CCC_1\oplus \CCC_2$ be the translational surface generated by the curves 
\[
\begin{array}{l}
\CCC_1: t\longmapsto \bfc_1(t) = \left(\dfrac{t}{t^6+1},\dfrac{t^3}{t^6+1},\dfrac{t^5}{t^6+1}\right), \\
\CCC_2:s \longmapsto \bfc_2(s) = \left(s^2,\dfrac{s}{s^2+1},s^4-2\right).
\end{array}
\]
Furthermore, for $i=1,2$ we consider $\bff_i(\bfx)=\bfA\bfx+\bfb_i$, with 
\[
\bfA=\begin{bmatrix} 
\sqrt{2} & -1 & 2\sqrt{2}\\
1 & -2 & \sqrt{2}\\
-1 & -\sqrt{2} & 2
\end{bmatrix}, \qquad \bfb_1={\bf 0}, \qquad \bfb_2=
\begin{bmatrix} 1 \\ 0 \\ -1 \end{bmatrix},
\]
and we let $S_2=\DDD_1\oplus\DDD_2$ be the translational surface where $\DDD_i=\bff_i(\CCC_i)$ for $i=1,2$, parametrized by $\bfd_1(t)=(\bff_1\circ\bfc_1)(t-1)$, $\bfd_2(s)=(\bff_2\circ\bfc_2)(s-1)$. 

According to Corollary \ref{cor1}, the surfaces $S_1, S_2$ are affinely equivalent and satisfy $S_2=\bff(S_1)$, where 
\begin{equation}\label{foreseen}
\bff(\bfx)=\bfA\bfx+\bfb,\qquad \bfb=\bfb_1\oplus\bfb_2=
\begin{bmatrix}
1/2\\ 0 \\ -1/2
\end{bmatrix}.
\end{equation}
Using Maple 18 \cite{Maple}, we compute the affinities $\bff_{ij}$ between the $\CCC_i$ and the $\DDD_j$. For $i\neq j$ no affinity is found, while for $i=j$ we get 
\[
\bff_{11}(\bfx)=\bfA\bfx,\qquad \bff_{22}(\bfx)=\bfA\bfx+\bfb_2
\]
with corresponding M\"obius transformations
\[ \varphi_1(t)=t+1,\qquad \varphi_2(s)=s+1.\]
Thus, as expected we recover the affinity \eqref{foreseen} between $S_1$ and $S_2$. The interested reader can check \cite{Muntingh} for further detail on the computations.
\end{example}

\subsection{Designing symmetric translational surfaces}

We can apply the preceding results to construct symmetric translational surfaces. In order to do this, we observe that Corollary \ref{cor1} also holds when replacing ``affinely equivalent'' with ``isometric'', and ``affine equivalence/map'' with ``isometry''. In particular, Corollary \ref{cor1} provides two quick ways to generate a symmetric translational surface, which are summarized in the following result.

\begin{proposition} \label{sym-new}
Let ${\mathcal S}=\CCC_1\oplus \CCC_2$, and let $\bff:{\Bbb K}^3\to {\Bbb K}^3$ be a symmetry satisfying one of the following two conditions:
\begin{itemize}
    \item[(a)] $\bff$ is a common symmetry of $\CCC_1$ and $\CCC_2$;
    \item[(b)] $\bff$ is an involution, i.e., $\bff\circ \bff$ is the identity, and $\CCC_2=\bff(\CCC_1)$.
\end{itemize}
Then $\bff$ is a symmetry of ${\mathcal S}$.
\end{proposition}

We illustrate this result in the following examples. 

\begin{example}[Crunode]
Let $\CCC_1\subset \RR^3$ be the crunode curve parametrized by
\begin{equation}\label{eq:crunode}
\bfc_1(t)  = \big(x(t), y(t), z(t)\big) = \left( \frac{t}{t^4 + 1}, \frac{t^2}{t^4 + 1}, \frac{t^3}{t^4 + 1} \right).
\end{equation}
As shown in \cite{AHM15b}, this curve is invariant under the half-turn about the $y$-axis, as
\[
\begin{bmatrix}
-1 & 0 & 0\\ 0 & 1 & 0\\ 0 & 0 & -1
\end{bmatrix}
\begin{bmatrix} x(t)\\ y(t)\\ z(t) \end{bmatrix} = 
\begin{bmatrix} x(-t)\\ y(-t)\\ z(-t) \end{bmatrix}.
\]
as well as reflections in the planes $z \pm x = 0$, since
\[
\begin{bmatrix}
0 & 0 & \pm 1\\ 0 & 1 & 0\\ \pm 1 & 0 & 0
\end{bmatrix}
\begin{bmatrix} x(t)\\ y(t)\\ z(t) \end{bmatrix} = 
\begin{bmatrix} \pm z(t)\\ \phantom{+} y(t)\\ \pm x(t) \end{bmatrix} = 
\begin{bmatrix} x(\pm 1/t)\\ y(\pm 1/t)\\ z(\pm 1/t) \end{bmatrix}.
\]
We consider now the curves $\bfc_2$ parametrized as 
 $\bfc_2 = \bfc_1$, $\bfc_2 = \bfM_y \bfc_1$ and $\bfc_2 = \bfM_z \bfc_1$, where the matrices $\bfM_y$ and $\bfM_z$ denote reflections in the planes $y=0$ and $z=0$. Proposition \ref{sym-new} implies that the translational surfaces $\bfc_1\oplus \bfc_2$ are symmetric; these surfaces are shown in the first row of Table \ref{tab:surfaces}.
\end{example}

\begin{example}[Trefoil knot]
Let $\CCC_1 \subset \RR^3$ be the trefoil knot parametrized by
\begin{equation}\label{eq:trefoilknot}
\bfc_1(\theta) = \big(x(\theta), y(\theta), z(\theta)\big) := \big( \sin(\theta) + 2\sin(2\theta), \cos(\theta) - 2\cos(2\theta), -\sin(3\theta) \big).
\end{equation}
Since the parametrization is trigonometric (i.e., each component is polynomial in $\sin(\theta)$ and $\cos(\theta)$), the reparametrization $z=e^{i\theta}$ provides a rational, complex parametrization $\widehat{\bfc}_1:S^1 \longrightarrow \CCC_1$, with $S^1 := \{z\in \CC\,:\,|z|=1\}$ the unit circle. Then, applying the method in \cite{AQ2} (cf. \cite{Muntingh}),
one obtains rotational symmetries about the $z$-axis by angles $\pm 2\pi/3$, as well as a half-turn about the $y$-axis. Alternatively, these symmetries are determined directly by applying a rotation matrix and standard trigonometric identities. In particular,
\[
\begin{bmatrix}
\cos\varphi & -\sin\varphi\\\sin\varphi & \cos\varphi
\end{bmatrix}
\begin{bmatrix} x(\theta)\\ y(\theta) \end{bmatrix} = 
\begin{bmatrix}
\sin(\theta - \varphi) + 2\sin(2\theta + \varphi)\\
\cos(\theta - \varphi) - 2\cos(2\theta + \varphi
\end{bmatrix}
=
\begin{bmatrix}
x(\theta - \varphi)\\ y(\theta - \varphi)
\end{bmatrix}
\]
holds identically if and only if $\varphi\equiv 0$ or $\varphi \equiv \pm 2\pi/3$ modulo $2\pi$, while
\[
\begin{bmatrix}
-1 & 0 & 0\\ 0 & 1 & 0\\ 0 & 0 & -1
\end{bmatrix}
\begin{bmatrix} x(\theta)\\ y(\theta)\\ z(\theta) \end{bmatrix} = 
\begin{bmatrix} x(-\theta)\\ y(-\theta)\\ z(-\theta) \end{bmatrix}.
\]
Choosing $\bfc_2$ as in the previous example yields the symmetric translational surfaces $\bfc_1\oplus \bfc_2$ shown in the second row of Table \ref{tab:surfaces}.
\end{example}

\begin{table}[]
    \noindent\begin{tabular*}{\columnwidth}{@{\extracolsep{\stretch{1}}}*{4}{c}}    
    \toprule
    $\bfc_1$ & $\bfc_1\oplus \bfc_1$ & $\bfc_1\oplus \bfM_y\bfc_1$
     & $\bfc_1\oplus \bfM_z\bfc_1$ \\ \midrule
    Crunode \eqref{eq:crunode} \\
    \includegraphics[height=1.9cm]{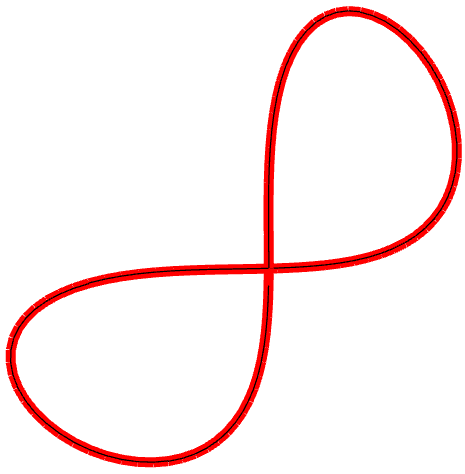} & \includegraphics[height=2.5cm]{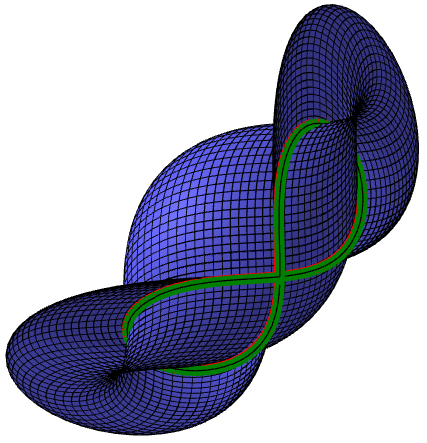} & 
    \includegraphics[height=2.5cm]{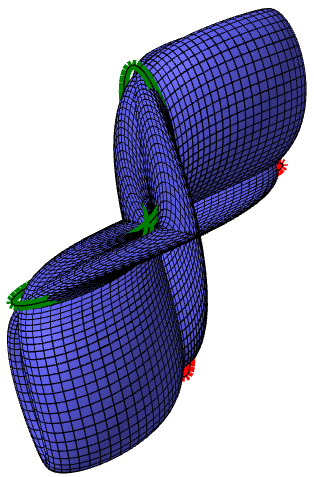} &
    \includegraphics[height=2.5cm]{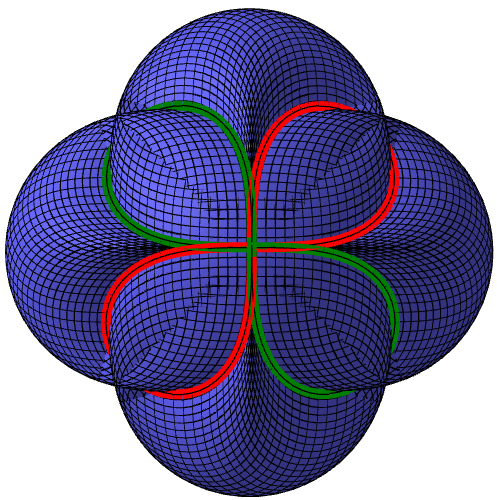} \\
    \midrule
    Trefoil knot \eqref{eq:trefoilknot} \\
    \includegraphics[height=1.9cm]{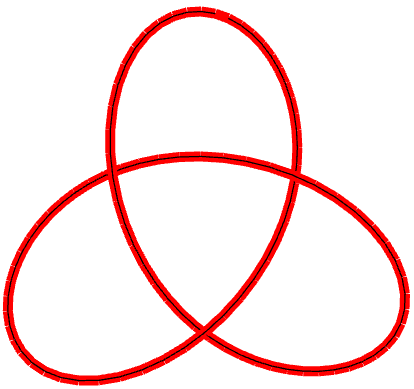} &
    \includegraphics[height=2.5cm]{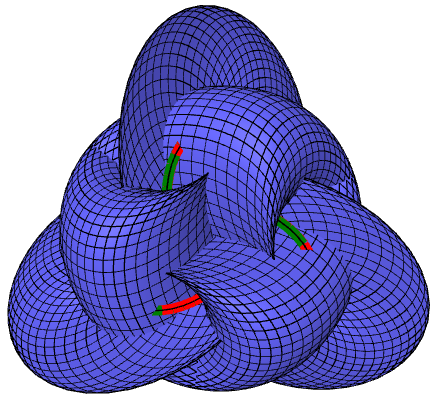} & 
    \includegraphics[height=2.5cm]{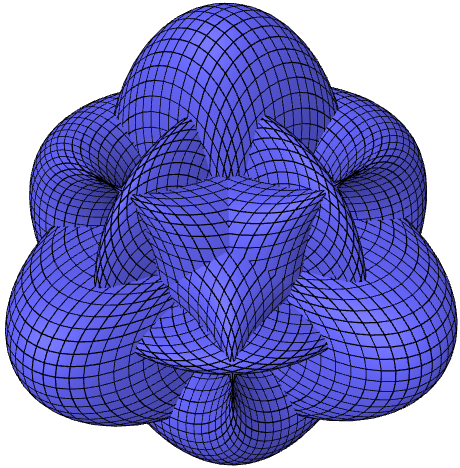} &
    \includegraphics[height=2.5cm]{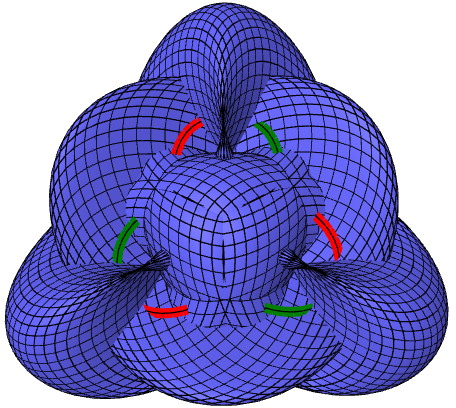} \\
    \bottomrule
    \end{tabular*}
    \caption{Translational surfaces inheriting symmetries shared by their generator pairs.}
    \label{tab:surfaces}
\end{table}

\section{Minimal surfaces}

Since minimal surfaces are translational surfaces with complex conjugate generator pairs (see Section \ref{sec:MinimalSurfaces1}) $({\mathcal C},\overline{\mathcal C})$, we can apply the previous results to minimal surfaces for which the minimal curve ${\mathcal C}$ is rational or meromorphic. In this case ${\mathcal S}$ is rational or meromorphic as well.

\subsection{Detecting affine equivalences and symmetries}\label{subsec-min}

In this subsection we consider minimal surfaces with parametrizations $\bfP$ as in \eqref{re1}, where $\bfc$ is rational or meromorphic. If $\bfP$ is rational, $\bfc$ must also be rational (see \cite[Corollary 22.25]{Gray}), and the parametrization \eqref{eq:re2} of the real part of $\SSS=\CCC\oplus \overline{\CCC}$ is rational in $u,v$. Note that while rational parametrizations $\bfc$ as in \eqref{weierstrass} come from rational pairs $f,g$, not every pair of rational functions $f,g$ provides a rational $\bfc$ \cite{VL17}. We will make the additional assumption that $\bfc$ is proper.

Recall that while we have presented an algorithm for detecting affine equivalences in the rational case, we do not have such an algorithm for the meromorphic case, since an algorithm for checking affine equivalence of meromorphic curves is currently absent. For this reason we will focus on the rational case, while making clear which results also hold in the meromorphic case.

Therefore, for $i=1,2$, given minimal surfaces $\SSS_i$ which are not multitranslational, rationally parametrized by $\bfP_i = \big(\bfc_i\oplus \overline{\bfc}_i\big) \circ \iota$ as in \eqref{eq:re2} with $\bfc_i$ proper, we can use Algorithm {\tt Affine-Equiv-Trans} to determine whether $\SSS_1,\SSS_2$ are affinely equivalent, by determining whether their minimal curves are affinely equivalent. As we only consider real affine equivalences, there is an additional advantage here: while the case of general surfaces of translation requires finding the affine equivalences between four pairs of space curves, this case only requires finding the affine equivalences between two pairs of space curves. Indeed, if there exist $\bfM\in \RR^{3\times 3}$, $\bfb\in \RR^3$ and $\varphi$ a M\"obius transformation satisfying
\[ \bfM \bfc_1(t)+\bfb=\bfc_2\circ \varphi(t), \]
conjugating this equation and substituting $s:=\overline{t}$ yields
\[ \bfM\overline{\bfc}_1(s)+\bfb=\overline{\bfc}_2 \circ \overline{\varphi}(s).
\]
Thus, if $\bfc_1$ and $\bfc_2$ parametrize complex space curves that are related by a real affine map, the same affine map relates the complex curves parametrized by $\overline{\bfc}_1$ and $\overline{\bfc}_2$ (although the corresponding M\"obius transformation is complex conjugated). A similar statement holds for $\bfc_1$ and $\overline{\bfc}_2$. Hence, we have the following result, which is another corollary of Theorem \ref{th-fundam-trans}. 

\begin{corollary}\label{cormin}
For $i=1,2$, let $\SSS_i$ be a minimal surface that is not multitranslational, rationally parametrized by $\bfP_i = (\bfc_i\oplus \overline{\bfc}_i)\circ \iota$ as in \eqref{eq:re2}, with $\bfc_i$ a proper parametrization of a minimal curve $\CCC_i$. Then $\bff$ is an affine equivalence between $\SSS_1$ and $\SSS_2$ if and only if, for some M\"obius transformation $\varphi$, one of the following cases holds:
\begin{enumerate}
\item $\bff \circ \bfc_1 = \bfc_2 \circ \varphi$\qquad ($\bff$ is an affine equivalence between $\CCC_1$ and $\CCC_2$)
\item $\bff \circ \bfc_1 = \overline{\bfc}_2 \circ \varphi$\qquad ($\bff$ is an affine equivalence between $\CCC_1$ and $\overline{\CCC}_2$)
\end{enumerate}
\end{corollary}

Next we consider the symmetries of a rational minimal surface $\SSS$, rationally parametrized by $\bfP = (\bfc \oplus \overline{\bfc})\circ \iota$ as in \eqref{eq:re2}. The following result follows directly from Corollary~\ref{cormin}. 

\begin{proposition} \label{syms-minimal}
Let $\SSS$ be a rational minimal surface that is not multitranslational, rationally parametrized by $\bfP = (\bfc\oplus \overline{\bfc})\circ \iota$ as in \eqref{eq:re2}, with $\bfc$ a proper parametrization of a minimal curve $\CCC$. Then $\bff$ is a symmetry of $\SSS$ if and only if, for some M\"obius transformation $\varphi$, either of the following cases holds:
\begin{enumerate}
\item $\bff \circ \bfc = \bfc \circ \varphi$\qquad ($\bff$ is a symmetry of $\CCC$)
\item $\bff \circ \bfc = \overline{\bfc} \circ \varphi$\qquad ($\bff$ is an isometry mapping $\CCC$ onto $\overline{\CCC}$)
\end{enumerate}
\end{proposition}

\begin{remark}\label{rem-ext-1}
Following Remark \ref{extension-meromorphic}, Corollary \ref{cormin} and Proposition \ref{syms-minimal} also hold when $\bfc$ is a proper parametrization with meromorphic components.
\end{remark}

We now introduce an additional assumption, namely that the functions $f,g$ defining $\bfc$ in \eqref{weierstrass} are rational functions with real coefficients. Many of the classical algebraic minimal surfaces found in the literature take this form (c.f. \cite{Ohdenal} and \cite[Chapter 22]{Gray}). Thus, let $\CCC\subset \CC^3$ be the complex curve parametrized by $\bfc = (x, y, z): U\subset \CC\longrightarrow \CC^3$ as in \eqref{weierstrass}. If the functions $f,g$ in \eqref{weierstrass} have real coefficients, conjugating \eqref{weierstrass} shows that the complex conjugate curve $\overline{\CCC}$ admits the parametrization 
\begin{equation}\label{parastar}
\overline{\bfc}(s) = \Big(\overline{x}(s),\overline{y}(s), \overline{z}(s)\Big) = \Big(x(s),-y(s),z(s)\Big),\quad s \in \overline{U}.
\end{equation}
For the trivial M\"obius transformation $\varphi(z)=z$ and reflection 
\begin{equation}\label{the-isom}
\bff(\bfx)=\bfM \bfx +\bfb, \qquad 
\bfM=\begin{bmatrix} 1 & 0 & 0 \\ 0 & -1 & 0 \\ 0 & 0 & 1 \end{bmatrix},\qquad
\bfb={\bf 0}
\end{equation}
in the plane $y=0$, the parametrization \eqref{parastar} yields
\[ \bff\circ \bfc=\overline{\bfc}\circ \varphi. \]
Hence Proposition \ref{syms-minimal} states that $\bff$ is a symmetry of the surface $\SSS$. Thus we recover the following known result, which reveals that minimal surfaces generated by minimal curves \eqref{weierstrass} constructed from real rational functions $f,g$ always have at least one mirror symmetry.

\begin{corollary} \label{oneplane}
Every minimal surface $\SSS$ parametrized by $\bfP$ as in \eqref{eq:re2} and \eqref{weierstrass}, with $f,g$ rational functions with real coefficients, is symmetric with respect to the plane $y=0$.
\end{corollary}

\subsection{Symmetries of higher-order Enneper surfaces}\label{sec:Enneper}
As an application of the results in the previous subsection, in this subsection we illustrate how Proposition \ref{syms-minimal} can be used to compute the symmetries of the \emph{(higher-order) Enneper surfaces} $\SSS_k$, for $k=1,2,\ldots$ (c.f. \cite{Karcher}). These are the minimal surfaces obtained by taking constant $f=2$ and monomial $g=z^k$ in \eqref{weierstrass}. The Enneper surfaces are classical examples of minimal surfaces with polynomial parametrizations.

We require an explicit parametrization of $\SSS_k$, which we derive due to lack of a suitable reference. The proof involves the \emph{Chebyshev polynomial} $T_n$ of the first kind, defined recursively by
\[ T_0(x) := 1,\qquad T_1(x) := x, \qquad T_n(x) := 2x T_{n-1}(x) - T_{n-2}(x),\qquad n\geq 2,\]
or implicitly by
\begin{equation}\label{eq:Timplicit1}
T_n\big(\cos(\theta)\big) = \cos(n\theta),\qquad n\geq 0.
\end{equation} 
Substituting $\theta = \frac{\pi}{2} - \theta'$ in \eqref{eq:Timplicit1}, one obtains
\begin{multline*}
T_n\big(\sin(\theta')\big) = T_n\big(\cos(\theta)\big) = \cos(n\theta)
= \cos\left(\frac{n\pi}{2} - n\theta' \right)\\ = \cos\left(\frac{n\pi}{2}\right)\cos(n\theta') + \sin\left(\frac{n\pi}{2}\right) \sin(n\theta'),
\end{multline*}
yielding the lesser-known identity
\begin{equation}\label{eq:Timplicit2}
T_n\big(\sin(\theta)\big) = (-1)^k \sin(n\theta),\qquad n = 2k+1 \geq 1.
\end{equation}

\begin{proposition}
For $k\geq 1$, the higher-order Enneper surface $\SSS_k$ admits the parametrization
\begin{equation}\label{eq:HigherEnneper}
\bfP_k(u,v)=\left( u - \frac{r^{2k+1} T_{2k+1}\left(\frac{u}{r}\right) }{2k+1}, -v - (-1)^k \frac{r^{2k+1}T_{2k+1}\left(\frac{v}{r}\right)}{2k+1}, 2 \frac{r^{k+1} T_{k+1} \left(\frac{u}{r}\right)}{k+1} \right).
\end{equation}
\end{proposition}

\begin{proof}
Write $t = re^{i\theta} = u + iv$. With $f(t)=2$ and $g(t) = t^k$, the expression \eqref{weierstrass} yields the minimal curve 
\begin{equation}\label{eq:minimalcurveEnneper}
\Psi_k(t)=\left(t - \frac{t^{2k+1}}{2k+1}, it + i\frac{t^{2k+1}}{2k+1}, 2\frac{t^{k+1}}{k+1} \right).
\end{equation}
From \eqref{eq:Timplicit1} it follows that
\[
\frac{t^n + \overline{t}^n}{2}
 = r^n \cos(n\theta) = r^n T_n \big( \cos(\theta) \big) = r^n T_n \left( \frac{u}{r}\right)\]
for $n\geq 0$, while \eqref{eq:Timplicit2} implies
\[
\frac{t^n - \overline{t}^n}{2i} = r^n \sin(n\theta) = (-1)^k r^n  T_n \big(\sin(\theta)\big) = (-1)^k r^n T_n \left(\frac{v}{r}\right)
\]
for $n=2k+1 \geq 1$. Hence the statement follows from $\bfP_k(u, v)=\Re\big(\Psi_k(u + iv)\big)$.
\end{proof}

\begin{example}
For $k=1$, we obtain the classical Enneper surface parametrized by
\begin{equation}
\bfP_1(u, v)=\left(v^2u-\frac13 u^3+u, \frac13v^3 - vu^2 - v, -v^2 + u^2\right),\qquad (u,v)\in \RR^2.
\end{equation}
All nontrivial minimal bicubic B\'ezier surfaces are affinely equivalent to pieces of this surface \cite[Theorem 2]{Cosin.Monterde02}; hence it is useful in computer-aided geometric design for the purpose of architecture, where minimal material usage is important.
\end{example}
\begin{remark}
With $r = \sqrt{u^2 + v^2}$ and $n = 2k + \varepsilon$ with $k\geq 0$ and $\varepsilon\in \{0, 1\}$, one can show by induction that
\begin{align*}
r^n T_n\left( \frac{u}{r} \right) & = \sum_{m = 0}^k (-1)^{k+m} {n\choose 2m + \varepsilon} v^{n - 2m - \varepsilon} u^{2m + \varepsilon}, \\
r^n T_n\left( \frac{v}{r} \right) & = \sum_{m = 0}^k (-1)^{k+m} {n\choose 2m + \varepsilon} u^{n - 2m - \varepsilon} v^{2m + \varepsilon}.
\end{align*}
This expresses the parametrization \eqref{eq:HigherEnneper} in the monomial basis.
\end{remark}

Let $O(3)$ be the \emph{orthogonal group} of $\RR^3$, i.e., the symmetry group of the sphere consisting of orthogonal $3\times 3$ matrices, and let
\begin{align*}
 D_{2k+2} & := \big\langle \rho, \sigma\,:\, \rho^{2k+2} = \sigma^2 = e,\ \sigma\rho\sigma = \rho^{-1}\big\rangle \\
          & \phantom{:} = \{\sigma^n \rho^m\,:\, n=0,1,\ m=0,\ldots,2k+1\}
\end{align*}
be the \emph{dihedral group of order $4k+4$} (here $e$ denotes the neutral element). Let
\begin{equation}\label{eq:RS}
\bfS := \begin{bmatrix}1&0&0\\0&-1&0\\0&0&1\end{bmatrix},\qquad
\bfR_k := 
\begin{bmatrix}
\phantom{+}\cos\big(\frac{\pi}{k+1}\big) & \sin\big(\frac{\pi}{k+1}\big) & 0\\
-\sin\big(\frac{\pi}{k+1}\big) & \cos\big(\frac{\pi}{k+1}\big) & 0\\
0 & 0 & -1
\end{bmatrix}.
\end{equation}

\begin{proposition}
The symmetry group $\{\bff_{m,n}(\bfx) := \bfM_{m,n}\bfx\}$ of the higher-order Enneper surface $\SSS_k$ is parametrized by the group monomorphism
\begin{equation}\label{eq:D4iso}
D_{2k+2} \longrightarrow O(3), \qquad
\sigma^n \rho^m \longmapsto \bfM_{m,n} := \bfS^n\bfR_k^m.
\end{equation}
Moreover, with $\bfP_k$ as in \eqref{eq:HigherEnneper} and M\"obius transformations $\varphi^m(z) := \zeta^m z$, where $\zeta = \zeta_{2k + 2} := e^{2\pi i/(2k+2)}$ is a $(2k+2)$-th root of unity,
\[ \bff_{m,n} \circ \bfP_k = \bfP_k \circ \varphi^m,\qquad n=0,1,\qquad m=0,1,\ldots,2k+1. \]
\end{proposition}
\begin{proof}
Applying Proposition \ref{syms-minimal} to compute the symmetries of $\SSS_k$, we first compute the symmetries of the complex space curve $\CCC_k$ parametrized by $\bfc_k$. Applying Proposition \ref{thm:diagram} with $\bfc=\bfd=\bfc_k$, each symmetry of $\CCC_k$ corresponds to an isometry $\bff(\bfx) = \bfM \bfx+\bfb$ and a M\"obius transformation $\varphi$ satisfying \eqref{af}. Since $\bfc_k$ is polynomial, $\varphi(t) = at + b$ is polynomial, and we obtain the polynomial system
\begin{equation}\label{eq:fundamentalPsi}
\bfM\bfc_k(t) + \bfb = \bfc_k(at + b).
\end{equation}
Writing $\bfM = [m_{ij}]_{ij}$ and $\bfb = [b_i]_i$, the last equation of this system is

\begin{multline}
(m_{31} + m_{32}i)t + \frac{2}{k+1} m_{33} t^{k+1} + \frac{1}{2k+1}(-m_{31} + i m_{32}) t^{2k+1} + b_3 \\
= \frac{2}{k+1} \sum_{j=0}^{k+1} {k+1 \choose j} a^jb^{k+1-j} t^j.
\end{multline}

Equating coefficients of (highest) order $2k + 1$ yields $m_{31} = m_{32} = 0$. Hence, since $\bfM$ is orthogonal, it follows that $m_{33} = \pm 1$. Equating coefficients of order $k + 1$ yields $a^{k+1} = m_{33} = \pm 1$, so that $a = \zeta^m$ for some $m\in \{0,1,\ldots,2k+1\}$. Moreover, equating linear coefficients yields $2ab^k = m_{31} + m_{32}i = 0$, implying $b=0$. Evaluating \eqref{eq:fundamentalPsi} at $t = 0$ yields $\bfb = \bfc_k(b) = {\bf 0}$. Differentiating \eqref{eq:fundamentalPsi} $l$ times and substituting $t=0$ yields
$\bfM \bfc_k^{(l)}(0) = a^{l} \bfc_k^{(l)}(0)$, which provides the matrix equation
\[ \bfM = \bfC \bfA_m \bfC^{-1},\]
where, since $\zeta_{2k+2}^{(2k+1)} = \zeta_{2k+2}^{-1}$ and $\zeta_{2k+2}^{k+1} = -1$,
\[
\bfC := \begin{bmatrix} \bfc_k'(0), \bfc_k^{(k+1)}(0), \bfc_k^{(2k+1)}(0)\end{bmatrix}, \qquad
\bfA_m := \begin{bmatrix} \zeta^m &0&0\\0& (-1)^m & 0\\0&0&\zeta^{-m}\end{bmatrix}.
\]
It follows that $\bfM_{m,0}\bfc_k = \bfc_k \circ \varphi^m$ for $m = 0,1,2,3$, where $\varphi^m(t) = \zeta^m t$ and 
\begin{align*}
\bfM_{m,0} & :=
\begin{bmatrix} 1 & 0 & -(2k)!\\ i & 0 & i(2k)!\\ 0 & 2 & 0 \end{bmatrix}
\begin{bmatrix} \zeta^m &0&0\\0& (-1)^m & 0\\0&0&\zeta^{-m}\end{bmatrix}
\frac12 \begin{bmatrix} 1 & -i & 0\\ 0 & 0 & 1\\ \frac{-1}{(2k)!} & \frac{-i}{(2k)!} & 0 \end{bmatrix}\\
 & \phantom{:} =
\begin{bmatrix}
\frac{\zeta^m + \zeta^{-m}}{2} & \frac{\zeta^m - \zeta^{-m}}{2i} &0\\
\frac{\zeta^{-m}-\zeta^m}{2i} & \frac{\zeta^m + \zeta^{-m}}{2} & 0\\
0 & 0 & (-1)^m \\
 \end{bmatrix}
 = 
\begin{bmatrix}
\phantom{+}\cos(\frac{\pi m}{k+1}) & \sin(\frac{\pi m}{k+1}) & 0\\
-\sin(\frac{\pi m}{k+1}) & \cos(\frac{\pi m}{k+1}) & 0\\
0 & 0 & (-1)^m
\end{bmatrix} \\  
& \phantom{:} = \bfR_k^m.
\end{align*}

Next we compute the isometries mapping $\CCC_k$ onto the complex curve $\overline{\CCC}_k$ parametrized by $\overline{\bfc}_k$. Applying Proposition \ref{thm:diagram} with $\bfc=\overline{\bfd}=\bfc_k$, each such isometry $\bff(\bfx) = \bfM \bfx+\bfb$ corresponds to a M\"obius transformation $\varphi$ satisfying \eqref{af}, again necessarily polynomial:
\begin{equation}\label{eq:fundamentalPsiOverline}
\bfM\bfc_k(t) + \bfb = \overline{\bfc}_k(at + b).
\end{equation}
Proceeding as before, one demonstrates that $\varphi(t) = \varphi^m(t) := \zeta^m t$ and $\bfb = {\bf 0}$. It follows that $\bfM_{m,1}\bfc_k = \overline{\bfc}_k\circ \varphi^m$ for $m = 0,1,\ldots,2k+1$, where 
\begin{multline}
\bfM_{m,1} = \overline{\bfC} \bfA_m \bfC^{-1} = 
\bfS \bfC \bfA_m \bfC^{-1} \\
=
\begin{bmatrix} 1&0&0\\0&-1&0\\0&0&1 \end{bmatrix}
\begin{bmatrix}
\phantom{+}\cos(\frac{\pi m}{k+1}) & \sin(\frac{\pi m}{k+1}) & 0\\
-\sin(\frac{\pi m}{k+1}) & \cos(\frac{\pi m}{k+1}) & 0\\
0 & 0 & (-1)^m
\end{bmatrix}
= \bfS \bfR_k^m.
\end{multline}

One verifies that the map $\sigma^n \rho^m \longmapsto \bfM_{m,n}$ is a monomorphism by comparing multiplication tables, or simply by verifying that its generators $\bfR_k$ and $\bfS$ satisfy $\bfR_k^{2k+2} = \bfS^2 = \bfI$ and $\bfS \bfR_k \bfS = \bfR_k^{-1}$.
\end{proof} 

\begin{remark}
Since $\bfS \bfR_k^m = \bfR_k^{-m} \bfS$, precomposing \eqref{eq:D4iso} with the group automorphism $\sigma^n \rho^m \longmapsto \sigma^n \rho^{-m}$ of $D_{2k+2}$ yields an alternative group monomorphism
\begin{equation*}
D_{2k+2} \longrightarrow O(3), \qquad
\sigma^n \rho^m \longmapsto \bfM_{-m,n} := \bfR_k^m \bfS^n.
\end{equation*}
\end{remark}

Note that $\bfM_{m,n}$ is a rotation of angle $-\frac{\pi m}{k+1}$ about the $z$-axis, composed by a reflection in the plane $z = 0$ when $m\equiv 1$ (mod 2), and in addition composed by a reflection in the plane $y = 0$ in the case $n=1$. For $k=1,2,3,4$, Figure \ref{fig:HigherEnneper} shows the higher-order Enneper surface $\SSS_k$, together with its symmetry elements.

The particular form $\varphi^m(t) := \zeta^m t$ of the M\"obius transformation was the inspiration for the results in the next subsection. 

\newcommand{\enp}[3]{\includegraphics[scale=0.39, clip=true, trim=30 45 30 40]{Enneper-k-#1-elements-elevation-#2-azimuth-#3-N-200.png}}
\begin{figure}
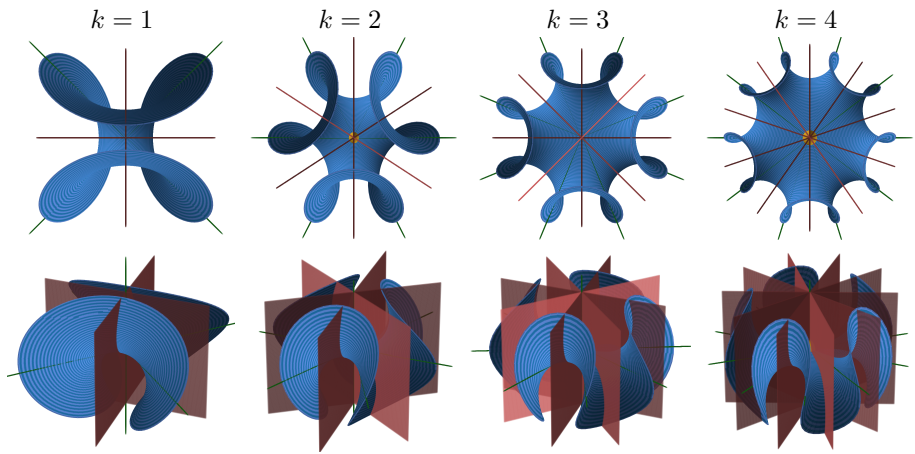

\bgroup
\setlength{\tabcolsep}{0.0em}
\begin{tabular}{cccc}
$k=1$ & $k=2$ & $k=3$ & $k=4$\\
\enp{1}{90}{0}  & \enp{2}{90}{0}  & \enp{3}{90}{0} & \enp{4}{90}{0}\\
\enp{1}{30}{20} & \enp{2}{30}{20} & \enp{3}{30}{20} & \enp{4}{30}{20}
\end{tabular}
\egroup

\caption{For $k=1,2,3,4$, top view (top) and side view (bottom) of higher-order Enneper surfaces $\SSS_k$, together with symmetry planes, symmetry rotation axes, and symmetry point (the latter for $k=2,4$).}\label{fig:HigherEnneper}
\end{figure}

\subsection{Constructing symmetric surfaces}\label{subsec-prescribed}

Inspired by Section \ref{sec:Enneper}, in this subsection we will see that imposing certain parity-like properties on the functions $f,g$ in \eqref{weierstrass} will result in a minimal surface $\SSS$ with certain symmetries. In this case, we need to make certain assumptions on $f,g$ (see Proposition \ref{prop:SchroderSymmetries}). When satisfied, these assumptions lead to meromorphic, not necessarily rational, parametrizations of $\SSS$.

Let $\FFF$ be the space of meromorphic functions on a simply connected region $U\subset\CC$ and let $\varphi: U\longrightarrow U$ be meromorphic. Consider the composition operator
\begin{equation}\label{eq:composition-operator}
T_\varphi: \FFF\longrightarrow\FFF, \qquad T_\varphi (f) := f\circ \varphi.
\end{equation}
The eigenvalue equation
\[ T_\varphi(f) = \lambda f \]
is called \emph{Schr\"oder's equation}; it is known to have solutions under general conditions.

In our case, for any integer $K\geq 2$, we consider the M\"obius transformation $\varphi(t) = \varphi_K(t) := \zeta_K \cdot t$, where $\zeta_K$ denotes a $K$-root of the unity. Notice that $\varphi$ leaves invariant any complex disk centered at the origin; in particular, in this case we can take $U$ to be any such disk, or even the entire complex plane $\CC$. The corresponding composition operator $T = T_\varphi$ generalizes the parity operator. The $K$-fold composition satisfies $T^K (f) = f\circ \varphi^K = f$, which implies that the eigenvalues of $T$ are the $K$-th roots of unity $\zeta_K^m$, with $m = 0,\ldots, K-1$. These provide an eigendecomposition of the function space $\FFF = \FFF_0 \oplus \cdots \oplus \FFF_{K-1}$ into $K$ parts. 

For simplicity we restrict ourselves to $K=4$, in which case $\zeta_K = i$; a similar analysis can be carried out for any $K\geq 2$. The following proposition states that choosing $f,g$ in \eqref{weierstrass} as eigenfunctions of $T$ (and hence of $T^q$, with $q\geq 1$) results in certain symmetries of the corresponding minimal surface. More precisely, we obtain a symmetry for every pair of eigenpairs $(i^r, f), (i^s, g)$ of $T^q$ for which $q + r + s\equiv 0$ (mod 2). 
The functions $f, g$ and domain $U$ satisfy the requirements for the Weierstrass representation \eqref{weierstrass}, described in more detail in Section \ref{sec:MinimalSurfaces1}.

\begin{proposition}\label{prop:SchroderSymmetries}
Let $\varphi(t) = i t$ and $T = T_\varphi$ be as above. Suppose that in a simply-connected region $U\subset \CC$ invariant under $\varphi$, $f$ is holomorphic, $g$ is meromorphic with no pole at $t=0$, and $fg^2$ is holomorphic. 
In addition, suppose $f,g$ satisfy
\begin{equation}\label{eq:parityconditions}
T^q(f)(t) = f(i^q t)=i^r \cdot f(t),\qquad T^q(g)(t) = g(i^q t)=i^s\cdot g(t)
\end{equation}
for some $q,r,s \in \ZZ/4\ZZ$ satisfying $q + r + s\equiv 0$ (mod 2). Let $\CCC$ be the corresponding curve parametrized by $\bfc$ as in \eqref{weierstrass}, with complex conjugate $\overline{\CCC}$ parametrized by $\overline{\bfc}$. Let $\SSS$ be the corresponding minimal surface parametrized by $\bfP$ as in \eqref{eq:re2}. Then $\SSS$ has the symmetry $\bff^\pm_{q+r,s}(\bfx) = \bfM^\pm_{q+r,s} \bfx$ as in Table \ref{tab:isometries}, for each choice of the sign $\pm$.
\end{proposition}

\begin{proof}
Suppose \eqref{eq:parityconditions} holds for some $q,r,s \in \ZZ/4\ZZ$ satisfying $q + r + s\equiv 0$ (mod~2). By Proposition \ref{syms-minimal} and Remark \ref{rem-ext-1}, the symmetries of $\SSS$ are the symmetries of $\CCC$ and the isometries mapping $\CCC$ onto $\overline{\CCC}$. In light of Proposition~\ref{thm:diagram} and Remark \ref{extension-meromorphic}, we examine when the reparametrization $\varphi^q(t) := i^q t$ of $\bfc$ can be expressed as the composition of such an isometry with either $\bfc$ or $\overline{\bfc}$. Applying the change of variable $\gamma=i^q \cdot \eta$ and using \eqref{eq:parityconditions},
\begin{align}
\Psi_1(i^q t) &  = \int_0^{i^q\cdot t} f(\gamma) \frac{1-g^2(\gamma)}{2} \rmd\gamma \label{psi1}
              && = i^{q+r} \int_0^t f(\eta) \frac{1-(-1)^s g^2(\eta)}{2} \rmd\eta, \\
\Psi_2(i^q t) &  = i\cdot\int_0^{i^q\cdot t} f(\gamma)\frac{1+g^2(\gamma)}{2}\rmd\gamma \label{psi2}
              \!\!\!\!\!\!\!&& = i^{q+r+1} \int_0^t f(\eta) \frac{1+(-1)^{s}g^2(\eta)}{2} \rmd\eta,\\
\Psi_3(i^q t) &  = \int_0^{i^q\cdot t} f(\gamma)g(\gamma) \rmd\gamma \label{psi3}
              && = i^{q+r+s} \int_{0}^{t} f(\eta)g(\eta) \rmd\eta.
\end{align}
With $t:=q+r$ and using that $\bfc = \bfS\overline{\bfc}$, it follows that
\begin{alignat*}{3}
\bfc \circ \varphi^q & = 
\begin{bmatrix}
i^t&0&0\\0&i^t & 0\\0&0&i^{t+s}
\end{bmatrix}
\bfc
&& = 
\begin{bmatrix}
i^t&0&0\\0&i^{t-2} & 0\\0&0&i^{t+s}
\end{bmatrix}
\overline{\bfc}
,\quad & \text{if } s\equiv 0\text{ (mod 2)},\\
\bfc \circ \varphi^q & = 
\begin{bmatrix}
0 & i^{t-1}&0\\i^{t+1} & 0 & 0\\0&0&i^{t+s}
\end{bmatrix}
\bfc && \, =
\begin{bmatrix}
0 & i^{t+1}&0\\i^{t+1} & 0 & 0\\0&0&i^{t+s}
\end{bmatrix}
\overline{\bfc}
,\quad & \text{if } s\equiv 1\text{ (mod 2)},
\end{alignat*}
so that
\begin{align}
\bfc \circ \varphi^q (t) =&  \bff^+_{q+r,s} \circ \bfc(t) = \bfM^+_{q+r,s} \bfc(t), \label{eq:Mplus}\\
\bfc \circ \varphi^q (t) =&  \bff^-_{q+r,s} \circ \overline{\bfc}(t) = \bfM^-_{q+r,s} \overline{\bfc}(t),\label{eq:Mminus}
\end{align}
where $\bfM^+_{q+r,s} = \bfM^-_{q+r,s} \bfS$ is as in Table \ref{tab:isometries}. The real isometries $\bff^\pm_{q+r,s}$ are obtained by discarding the cases $q + r + s\equiv 1$ (mod 2), shown in gray. Thus, from Proposition~\ref{thm:diagram} we deduce that in the remaining cases $\bff^+_{q+r,s}$ is a symmetry of $\CCC$ and $\bff^-_{q+r,s}$ is an isometry mapping $\CCC$ onto $\overline{\CCC}$. In either case, Proposition~\ref{syms-minimal} implies that $\bff$ is a symmetry of the surface $\SSS$.
\end{proof}

\begin{table}
{\small
\noindent\begin{tabular*}{\columnwidth}{c@{\extracolsep{\stretch{1}}}*{4}{c}}
\toprule
$\bfM^+_{q+r,s}$ & $q + r \equiv 0$ & $q + r \equiv 1$ & $q + r \equiv 2$ & $q + r \equiv 3$\\ \bottomrule
             & identity & & central inversion & \\ 
             & $\RR^3$ & & $x=y=z=0$ & \\
$s \equiv 0$ & $\begin{bmatrix} 1&0&0\\0&1&0\\0&0&1 \end{bmatrix}$ 
             & \color{Gray} $\begin{bmatrix} i&0&0\\0& i&0\\0&0&i \end{bmatrix}$ 
             & $\begin{bmatrix} -1&0&0\\0&-1&0\\0&0&-1 \end{bmatrix}$ 
             & \color{Gray} $\begin{bmatrix} -i&0&0\\0&- i&0\\0&0&-i \end{bmatrix}$ \\ \midrule
             & & rotoreflection & & quarter-turn \\
             & & $x=y=z=0$ & & $x=y=0$ \\
$s \equiv 1$ & \color{Gray} $\begin{bmatrix} 0&-i&0\\ i&0&0\\0&0&i \end{bmatrix}$ 
             & $\begin{bmatrix} 0&1&0\\-1&0&0\\0&0&-1 \end{bmatrix}$ 
             & \color{Gray} $\begin{bmatrix} 0&i&0\\-i&0&0\\0&0&-i \end{bmatrix}$ 
             & $\begin{bmatrix} 0&-1&0\\ 1&0&0\\0&0&1 \end{bmatrix}$ \\ \midrule
             & reflection & & half-turn & \\
             & $z = 0$ & & $x=y=0$ & \\
$s \equiv 2$ & $\begin{bmatrix} 1&0&0\\0& 1&0\\0&0&-1 \end{bmatrix}$ 
             & \color{Gray} $\begin{bmatrix} i&0&0\\0& i&0\\0&0&-i \end{bmatrix}$ 
             & $\begin{bmatrix} -1&0&0\\0&-1&0\\0&0&1 \end{bmatrix}$ 
             & \color{Gray} $\begin{bmatrix} -i&0&0\\0&-i&0\\0&0&i \end{bmatrix}$ \\ \midrule
             & & quarter-turn & & rotoreflection \\ 
             & & $x=y=0$      & & $x=y=z=0$      \\
$s \equiv 3$ & \color{Gray} $\begin{bmatrix} 0&-i&0\\ i&0&0\\0&0&-i \end{bmatrix}$ 
             & $\begin{bmatrix} 0&1&0\\ -1&0&0\\0&0&1 \end{bmatrix}$ 
             & \color{Gray} $\begin{bmatrix} 0&i&0\\ -i&0&0\\0&0&i \end{bmatrix}$ 
             & $\begin{bmatrix} 0&-1&0\\ 1&0&0\\0&0&-1 \end{bmatrix}$\smallskip \\ \toprule
$\bfM^-_{q+r,s}$ & $q + r \equiv 0$ & $q + r \equiv 1$ & $q + r \equiv 2$ & $q + r \equiv 3$\\ \bottomrule
             & reflection & & half-turn & \\
             & $y=0$      & & $x=z=0$   & \\
$s \equiv 0$ & $\begin{bmatrix} 1&0&0\\0& -1&0\\0&0&1 \end{bmatrix}$ 
             & \color{Gray} $\begin{bmatrix} i&0&0\\0& -i&0\\0&0&i \end{bmatrix}$ 
             & $\begin{bmatrix} -1&0&0\\0& 1&0\\0&0&-1 \end{bmatrix}$ 
             & \color{Gray} $\begin{bmatrix} -i&0&0\\0& i&0\\0&0&-i \end{bmatrix}$ \\ \midrule
             & & half-turn & & reflection \\
             & & $x+y = z = 0$ & & $x-y=0$ \\
$s \equiv 1$ & \color{Gray} $\begin{bmatrix} 0&i&0\\ i&0&0\\0&0&i \end{bmatrix}$ 
             & $\begin{bmatrix} 0&-1&0\\ -1&0&0\\0&0&-1 \end{bmatrix}$ 
             & \color{Gray} $\begin{bmatrix} 0&-i&0\\ -i&0&0\\0&0&-i \end{bmatrix}$ 
             & $\begin{bmatrix} 0&1&0\\ 1&0&0\\0&0&1 \end{bmatrix}$ \\ \midrule 
             & half-turn & & reflection & \\
             & $y=z=0$ & & $x=0$ & \\
$s \equiv 2$ & $\begin{bmatrix} 1&0&0\\0& -1&0\\0&0&-1 \end{bmatrix}$ 
             & \color{Gray} $\begin{bmatrix} i&0&0\\0& -i&0\\0&0&-i \end{bmatrix}$ 
             & $\begin{bmatrix} -1&0&0\\0& 1&0\\0&0&1 \end{bmatrix}$ 
             & \color{Gray} $\begin{bmatrix} -i&0&0\\0& i&0\\0&0&i \end{bmatrix}$ \\ \midrule
             & & reflection & & half-turn\\
             & & $x + y = 0$ & & $ x - y = z = 0$\\
$s \equiv 3$ & \color{Gray} $\begin{bmatrix} 0&i&0\\ i&0&0\\0&0&-i \end{bmatrix}$ 
             & $\begin{bmatrix} 0&-1&0\\ -1&0&0\\0&0&1 \end{bmatrix}$ 
             & \color{Gray} $\begin{bmatrix} 0&-i&0\\ -i&0&0\\0&0&i \end{bmatrix}$ 
             & $\begin{bmatrix} 0&1&0\\ 1&0&0\\0&0&-1 \end{bmatrix}$ \\
\bottomrule
\end{tabular*}
}
\caption{Real orthogonal (black) and imaginary unitary (gray) matrices $\bfM^\pm_{q+r,s}$ in \eqref{eq:Mplus}--\eqref{eq:Mminus} with symmetry types and symmetry elements for the various cases $(q,r,s)$, where $\equiv$ denotes equivalence modulo 4.}\label{tab:isometries}
\end{table}

The Enneper surface $\SSS_1$ originates from taking $f(t)=2$ and $g(t)=t$, in which case
\[ f(i^q t)=2,\qquad g(i^q t)=i^q t, \qquad q\in \ZZ/4\ZZ.\]
Hence \eqref{eq:parityconditions} holds whenever $r\equiv 0$ modulo 4 and $s\equiv q\equiv q + r$ modulo 4. Therefore the symmetries $\bfR_1^m \bfS^n$ of the Enneper surface are recovered as the diagonal cases $m \equiv s\equiv q+r$, with $n=0$ for the top sign and $n=1$ for the bottom sign.

\begin{remark}
Consider the (external) direct product group
\[ D_4 \times \ZZ/2\ZZ \simeq \langle \rho,\sigma,\tau\,:\,\rho^4=\sigma^2=\tau^2=e,\ \sigma\rho\sigma=\rho^{-1},\ \tau\rho=\rho\tau,\ \sigma\tau=\tau\sigma \rangle, \]
where $e$ denotes the neutral element. With $\bfR_1, \bfS$ as in \eqref{eq:RS} and with $\bfT := \diag(1,1,-1)$ the reflection in the plane $z=0$, the map
\[
D_4 \times \ZZ/2\ZZ \longrightarrow O(3), \qquad \rho^m \sigma^n \tau^p\longmapsto \bfR_1^m \bfS^n \bfT^p
\]
is a group monomorphism establishing a group structure on the set of real matrices in Table \ref{tab:isometries}.
\end{remark}

\begin{remark}
Alternatively, consider M\"obius transformations $\varphi(t) = t + b$ and $T=T_\varphi$ as in \eqref{eq:composition-operator} for a space $\FFF$ of periodic or doubly periodic meromorphic functions. For $K\geq 2$, choosing $b = \omega / K$ for one of the periods $\omega$ of $\FFF$, the operator $T$ again has order $K$ and eigenvalues $\zeta_K^m$, for $m=0,\ldots, K-1$. Analogous to Proposition \ref{prop:SchroderSymmetries}, solutions $f,g$ to Schr\"oder's equation again lead to symmetric minimal surfaces (c.f. \cite{Karcher}).
\end{remark}

\section{Conclusion and open problems}

In this paper we have introduced a characterization for affine equivalence of two surface of translation, defined by either rational or meromorphic generators. Since minimal surfaces are surfaces of translation with a complex conjugate generator pair, the results naturally translate to minimal surfaces of the considered kind as well. When the generators are rational, our algorithm leads to an algorithm that ultimately relies on the algorithm in \cite{HJ18} to check whether two space curves are affinely equivalent.  Additionally, we have applied our results to building surfaces of translation and minimal surfaces with symmetries, and to computing the symmetries of higher-order Enneper surfaces.

However, notice that the algorithms in this paper require the surfaces to be defined by means of certain types of parametrization. In the case of surfaces of translation, we need them to be given in the standard form $\bfP(u, v)=\frac12\big(\bfc_1(u)+\bfc_2(v)\big)$, where $\bfc_1(u)$ and $\bfc_2(v)$ are rational curves. In the case of minimal surfaces, we require them to be given as in \eqref{re1}, which in turn requires to know a minimal curve for the surface. 

If a surface of translation is reparametrized, then the standard form is lost. In the general case, it is still an open problem to efficiently recognise a surface as a surface of translation when it is not parametrized in the standard way (c.f. \cite[\S 2.3]{VL17}), and to bring it into standard form.
Similarly, if a rational minimal surface undergoes a rational reparametrization, computing a minimal curve for the surface is still an open problem.

A first step in this direction is the paper \cite{PDS14}, where ideas for recognizing rational surfaces of translation are provided. However, the algorithm in \cite{PDS14} requires computing the implicit equation of the surface and, which is harder, a tangent direction to one of the curves in the generator pairs. Although \cite{PDS14} presents a novel approach in this regard, a complete and efficient solution to the problem is still absent. However, since minimal surfaces are surfaces of translation with a complex conjugate generator pair, these two open problems, i.e., recognising and reparametrizing translational surfaces and computing a minimal curve, are certainly connected.

%\section*{References} 


\begin{thebibliography}{56}

\bibitem{AHM15} Alc\'azar J.G., Hermoso C. (2016), {\it Involutions of polynomially parametrized surfaces}, Journal of Computational and Applied Mathematics Vol. 294, pp. 23--38.

\bibitem{AHM14} Alc\'azar J.G., Hermoso C., Muntingh G. (2014), {\it Detecting similarity of rational plane curves}, Journal of Computational and Applied Mathematics vol. 269, pp. 1-13. 

\bibitem{AHM15b} Alc\'azar J.G., Hermoso C., Muntingh G. (2015), {\it Symmetry detection of rational space curves from their curvature and torsion}, Computer Aided Geometric Design Vol. 33, pp. 51--65.

\bibitem{ADM} Alc\'azar J.G., Dahl H.E.I., Muntingh G. (2018), {\it Symmetries of canal surfaces and Dupin cyclides}, Computer Aided Geometric Design Vol. 59, pp. 68--85.

\bibitem{AQ} Alc\'azar J.G., Quintero E. (2020), {\it Affine equivalences, isometries and symmetries of ruled rational surfaces}, Journal of Computational and Applied Mathematics Vol. 364. 

\bibitem{AQ2} Alc\'azar J.G., Quintero E. (2020), {\it Affine Equivalences of Trigonometric Curves}, Acta Applicandae Mathematicae Vol 170, pp. 691--708.

\bibitem{BLV} Bizzarri M., L\`avi\u cka M., Vr\u sek J. (2020), {\it Computing projective equivalences of special algebraic varieties}, Journal of Computational and Applied Mathematics Vol. 367. 

\bibitem{Chern} Chern S. (1982), {\it Web Geometry}, Bulletin of the American Mathematical Society, Vol. 6, No. 1, pp. 1--8. 

\bibitem{Cosin.Monterde02} Cos{\'i}n C.
and Monterde J. (2002), {\it B\'ezier surfaces of minimal area}. In: Sloot P.M.A., Hoekstra A.G., Tan C.J.K., Dongarra J.J. (eds) Computational Science --- ICCS 2002. ICCS 2002. Lecture Notes in Computer Science, vol 2330. Springer, Berlin, Heidelberg

\bibitem{Coxeter69} Coxeter H. (1969), {\it Introduction to geometry}, John Wiley \& Sons.

\bibitem{DHKW} Dierkes U., Hildebrandt S., K\"uster A., Wohlrab O. (1992). {\it Minimal surfaces}. vol. I. Springer.

\bibitem{Forster} Forster, O. (2013), {\it Riemannsche Fl\"achen}. Vol. 184. Springer-Verlag

\bibitem{Gray} Gray, A. (1999), {\it Modern Differential Geometry of Curves and Surfaces with Mathematica}, CRC Press. 

\bibitem{HJ18} Hauer M., J\"uttler B. (2018), {\it Projective and affine symmetries and equivalences of rational curves in arbitrary dimension}, Journal of Symbolic Computation Vol. 87, pp. 68--86.

\bibitem{HJ18-2} Hauer M., J\"uttler B., Schicho J. (2018), {\it Projective and affine symmetries and equivalences of rational and polynomial surfaces}, Journal of Computational and Applied Mathematics, Vol. 349, pp. 424--437.  
    
\bibitem{Karcher} Karcher H. (1989) {\it Construction of minimal surfaces}. Surveys in Geometry, 1--96, 1989. University of Tokyo, 1989, and Lecture Notes No. 12, SFB 256, Bonn.
    
\bibitem{Maple} Maple 18. Maplesoft, a division of Waterloo Maple Inc., Waterloo, Ontario. \texttt{http://www.maplesoft.com/} 

\bibitem{Muntingh} Muntingh, G. (2020) \texttt{https://github.com/georgmuntingh/min-surf/}

\bibitem{PDS14} P\'erez-D\'{\i}az S., Shen L-Y. (2020), 
Theoretical Computer Science, Vol. 835. pp. 156–-167. 

\bibitem{Ohdenal} Ohdenal B. (2016), {\it On algebraic minimal surfaces}, KoG Vol. 20, pp. 61--78. 

\bibitem{SWP} Sendra J.R., Winkler F., P{\'e}rez-D{\'{\i}}az S. (2008), {\it Rational algebraic curves}, Springer. 
   
\bibitem{Struik} Struik D.J. (1988), {\it Lectures on Classical Differential Geometry}, Dover publications.

\bibitem{VL17} Vr\u{s}ek J., L\'avi\u{c}ka, M. (2016), {\it Translation Surfaces and Isotropic Transport Nets on Rational Minimal Surfaces}, Proceedings of the congress Mathematical Methods for Curves and Surfaces (MMCS) 2016, Lecture Notes in Computer Science, vol 10521. Springer.  

\bibitem{WG18a} Wang H., Goldman R. (2018), {\it Syzygies for translational surfaces}, Journal of Symbolic Computation Vol. 89, pp. 73--93.

\bibitem{WG18b} Wang H., Goldman R. (2018), {\it Implicitizing ruled translational surfaces}, Computer Aided Geometric Design Vol. 59, pp. 98--106.
	
\end{thebibliography}
\end{document}